\documentclass[12pt]{amsart}
\usepackage{amsfonts}

\newtheorem{theorem}{Theorem}[section]
\newtheorem{corollary}[theorem]{Corollary}

\newtheorem{lemma}[theorem]{Lemma}

\newtheorem{proposition}[theorem]{Proposition}
\newtheorem{question}[theorem]{Question}
\newtheorem{example}[theorem]{Example}

\newenvironment{sketch}{\paragraph{\textit{Sketch of the proof.}}}{\hfill$\square$}

\def\NN{\hbox{\sf I\kern-.13em\hbox{N}}}
\def\RR{\hbox{\sf I\kern-.14em\hbox{R}}}

\def\Cc{\hbox{\sf C\kern -.47em {\raise .48ex \hbox{$\scriptscriptstyle |$}}
   \kern-.5em {\raise .48ex \hbox{$\scriptscriptstyle |$}} }}

\newcommand{\be}{\begin{equation}}
\newcommand{\ee}{\end{equation}}

\DeclareMathOperator{\MOD}{mod}
\DeclareMathOperator{\spam}{span}

\begin{document}

\title[Pairs of semi-commuting matrices]{On the dimension of the algebra generated by two positive semi-commuting matrices}
\author{Marko Kandi\'{c}, Klemen \v Sivic}
\date{\today}

\begin{abstract}
\baselineskip 7.5 mm
Gerstenhaber's theorem states that the dimension of  the unital algebra generated by two commuting $n\times n$ matrices is at most $n$. We study the analog of this question for positive matrices with a positive commutator. We show that the dimension of the unital algebra generated by the matrices is at most $\frac{n(n+1)}{2}$ and that this bound can be attained. We also consider the corresponding question if one of the matrices is a permutation or a companion matrix or both of them are idempotents. In these cases, the upper bound for the dimension can be reduced significantly. In particular, the unital algebra generated by two semi-commuting positive idempotent matrices is at most $9$-dimensional. This upper bound can be attained.
\end{abstract}

\maketitle

\noindent
{\it Math. Subj. Classification (2010)}: 15A27, 15B48, 47B47. 

\baselineskip 6mm

\section {Introduction}

A classical question in linear algebra asks for the upper bound of the dimension of a commutative algebra of $n\times n$ matrices. The basic question was answered by Schur \cite{Schur} showing that the upper bound for the dimension is $\lfloor \frac{n^2}{4}\rfloor +1$. An important version of the above question asks for the upper bound of the dimension of a unital algebra generated by two commuting matrices. Using algebraic geometry, Gerstenhaber proved the following well-known result.

\begin{theorem}[\cite{Gerst}]\label{Gerstenhaber}
If $n\times n$ matrices $A$ and $B$ commute, then  the unital algebra generated by $A$ and $B$ is at most $n$-dimensional.
\end{theorem}

As pointed out by Guralnick \cite{Guralnick}, the above theorem follows also from the irreducibility of the variety of all pairs of commuting $n\times n$ matrices, a result which was first proved by Motzkin and Taussky \cite{MotTau}. All these proofs of Theorem \ref{Gerstenhaber} use algebraic geometry. Purely linear-algebraic proofs, using generalized Cayley-Hamilton theorem, were provided by Barria and Halmos \cite{BarHal} and by Laffey and Lazarus \cite{LafLar}. We note that Theorem \ref{Gerstenhaber} fails for commutative algebras generated by more than 3 elements \cite{Guralnick}, while the question whether the dimension of an algebra generated by three commuting $n\times n$ matrices is bounded by $n$ is still open. See \cite{HolbrookO'Meara} for a recent approach to this problem, and references therein.

In the case of positive matrices it is natural to consider positivity of the commutator. The study of positive commutator of positive matrices and positive operators on Banach lattices was initiated in \cite{BDFRZ}. Such commutators have interesting properties (see \cite{BDFRZ}, \cite{Gao}, \cite{Roman}). For example, a positive commutator $[A,B]=AB-BA$ of positive matrices $A$ and $B$ is nilpotent and contained in the radical of the (Banach) algebra generated by $A$ and $B$. Furthermore, if one of the matrices $A$ and $B$ is ideal-irreducible, then positivity of their commutator implies that they actually commute. In this paper we connect the study of positive commutators of positive matrices with Gersenhaber's theorem. More precisely, we consider the following question.

\begin{question}\label{vprasanje}
Let $A$ and $B$ be positive matrices with a positive commutator $AB-BA.$ What is the upper bound for the dimension of the unital algebra generated by $A$ and $B$?
\end{question}

We prove that, in general, this dimension is at most $\frac{n(n+1)}{2}$. Then we consider Question \ref{vprasanje} for special types of matrices, where the upper bound can be reduced significantly.

The paper is organized as follows. In Section 2 we gather relevant notation, definitions and properties  needed throughout the text.

In Section 3 we answer the general form of Question \ref{vprasanje}. We show that the upper bound for the dimension in the question is $\frac{n(n+1)}{2}$. Furthermore, if one of the matrices in the question is ideal-irreducible, then the matrices commute, so that  the conclusion is in this case the same as in Gerstenhaber's theorem. We also give an example of two positive semi-commuting matrices where the upper bound $\frac{n(n+1)}{2}$ is attained.

In Section 4 we consider Question \ref{vprasanje} when one of the matrices is a permutation matrix. We prove that in this case the upper bound is again $n$. Along the way we completely describe the vector space  of matrices that intertwine two cycles of possibly different sizes, and we also determine its dimension.

In Section 5 we confine ourselves to the case when one of the matrices is a companion matrix. In this case, the upper bound in Question \ref{vprasanje} depends on the algebraic multiplicity of zero as an eigenvalue of the companion matrix. We also prove that only upper-triangular matrices can semi-commute with a given Jordan block.

In Section 6 we consider Question \ref{vprasanje} for the case of two positive idempotent matrices. When one matrix is assumed to be strictly positive, then the upper bound is $6$, and if we additionally assume that its transpose is strictly positive  as well, then the matrices commute and the upper bound is $4$. In general, we show that in an associative algebra two (not necessarily positive) idempotents $E$ and $F$ satisfying $(EF-FE)^n=0$ generate a unital algebra of dimension at most $4n$. This upper bound can be significantly reduced in the case when $E$ and $F$ are complex idempotent $n\times n$ matrices. Gaines, Laffey and Shapiro \cite{GainesLaffeyShapiro} proved that, in this case, the algebra is at most $2n$-dimensional if $n$ is even, and at most $(2n-1)$-dimensional if $n$ is odd. Moreover, this result holds also without the assumption on nilpotency of the commutator.
In the case of positive semi-commuting idempotent matrices, we show that this bound is actually $9$. The latter result holds also in a more general setting of vector lattices.

\section {Preliminaries}

Most of our results are obtained for matrices but some of them
hold in more general setting of vector lattices. Therefore, we first
provide basic definitions and properties of vector lattices.

An {\it ordered vector space} $L$ is a real vector space equipped with an order relation $\leq$ which is compatible with the vector space structure. If $L$ is a lattice with respect to the ordering $\leq$, then $L$ is said to be a {\it vector lattice} or a {\it Riesz space}.
Vector $x$ is called {\it positive} if $x\geq 0$. The set of all positive vectors is denoted by $L^+$ and is called the {\it positive cone} of $L$. The supremum and the infimum of vectors $x$ and $y$ are denoted by $x\vee y$ and $x\wedge y$, respectively.  The vector $x\vee -x$ is called the {\it modulus} or the {\it absolute value} of $x$, and is denoted by $|x|$.
An important example of a vector lattice is $n$-dimensional real vector space $\mathbb R^n$ equipped with the componentwise ordering. Therefore, if $x=(x_1,\ldots,x_n)$ is an arbitrary vector in $\mathbb R^n$, then $|x|=(|x_1|,\ldots,|x_n|)$. Positive vectors in $\mathbb R^n$ are precisely the ones with nonnegative coordinates. Therefore, they are sometimes called
{\it nonnegative vectors}.
A vector $x\in \mathbb R^n$ with positive coordinates is said to be {\it strictly positive}.

An {\it order ideal} of a vector lattice $L$ is a vector subspace $\mathcal{J}$ of $L$ with the property that conditions $0\leq |y|\leq |x|$ and $x\in \mathcal{J}$ imply $y\in \mathcal{J}$. In $\mathbb R^n$ the ideals are precisely standard subspaces, i.e.,
$\mathcal J$ is an order ideal in $\mathbb R^n$ whenever
$$\mathcal{J}=\{(x_1,\ldots,x_n)\in \mathbb R^n:\, x_i=0 \textrm{ for all } i\in \mathcal I\}$$ for some subset $\mathcal I$ of $\{1,\ldots,n\}.$ If $\mathcal I=\emptyset$, then $\mathcal{J}=\mathbb R^n.$
An order closed ideal is called a {\it band}. The {\it disjoint complement} $\mathcal S^d$ of a nonempty set $\mathcal S$ is defined as
$$\{x\in L:\; |x|\wedge |s|=0\quad \textrm{for all }s\in \mathcal S\}.$$ The set $\mathcal S^d$ is always a band. A band $\mathcal B$ in $L$ is called a {\it projection band} whenever
$L=\mathcal B\oplus \mathcal B^d.$ The latter direct sum is an order direct sum, i.e., if $x=y+z$ is a vector in $L$ for some $y\in \mathcal B$ and $z\in \mathcal B^d$, then $x$ is positive if and only if $y$ and $z$ are positive.
Vector lattice $L$ is said to have the {\it projection property} whenever every band in $L$ is a projection band.
It is well-known that
$\mathbb R^n$ has the projection property and that
every standard subspace of $\mathbb R^n$ is a projection band.

If $\|\cdot\|$ is a norm on $L$, then the pair $(L,\|\cdot\|)$ is called a {\it normed vector lattice} or a {\it normed Riesz space} whenever $0\leq |x|\leq |y|$ implies $\|x\|\leq \|y\|.$
A normed vector lattice that is also a Banach space is called a {\it Banach lattice}. A linear operator $T:L\to L$ is said to be {\it positive} whenever
$T(L^+)\subseteq L^+$.
We will write $T\geq 0$ whenever $T$ is positive.
An operator $T$ is said to be {\it negative} if $-T$ is positive.
If the commutator $AB-BA$ of operators $A$ and $B$
is positive or negative, then $A$ and $B$ are said to
{\it semi-commute}.

If $L=\mathbb R^n$ is equipped with the componentwise ordering, a linear operator $T$ on $L$ is positive if and only if the matrix that represents $T$ with respect to the standard basis has nonnegative entries. In matrix theory  matrices with nonnegative entries are called {\it nonnegative matrices}. However, in this paper we adopt the terminology that is usually used in the theory of vector lattices. Therefore the matrices with nonnegative entries will be called {\it positive matrices}. In the case of real matrices $A$ and $B$ we will write
$A\geq B$ whenever $A-B$ is
positive. It should be noted that $A\geq B$ if and only if
$a_{ij}\geq b_{ij}$ for all $i,j$.

The {\it absolute kernel} $\mathcal N(T)$ of a positive linear operator $T$ on a vector lattice $L$ is defined as the set $\{x\in L:\;T|x|=0\}.$ If $\mathcal N(T)=\{0\}$, then $T$ is said to be {\it strictly positive}. Therefore, a positive operator $T$ is strictly positive if and
only if the kernel $\ker T$ does not contain positive vectors.
A positive operator $T$ is called {\it order continuous} if for every decreasing net $\{x_\alpha\}_\alpha$ with zero as its infimum, the infimum of the net $\{Tx_\alpha\}_\alpha$ is also zero. It is a standard fact that the absolute kernel of an order continuous operator is always a band.

The set of all bounded linear functionals of a Banach lattice $L$ is well-known to be a Banach space. The Banach space $L^*$ can be naturally equipped with the ordering given by $\varphi\leq \psi$ in $L^*$ if and only if $\varphi(x)\leq \psi(x)$ for each $x\in L^+$. The dual Banach space $L^*$ equipped with this ordering becomes a Banach lattice. Every positive operator on a Banach lattice is bounded \cite[Theorem 4.2]{AB06}. For the details about vector and Banach lattices and operators acting on them we refer the reader to \cite{AbAp} and \cite{AB06}.

If not otherwise stated, matrices appearing throughout the
text are complex matrices of
size $n\times n$.
The spectral radius of a positive matrix is always an eigenvalue having a positive eigenvector (\cite[Theorem 8.11]{AbAp}). A matrix $A$ is {\it ideal-reducible} if there exists a nontrivial ideal of $\mathbb R^n$  that is invariant under $A$. Therefore, the matrix $A$  is ideal-reducible if and only if there exists a permutation matrix $P$ such that the matrix $P^TAP$ is of the form
$$
\left[
\begin {array}{cc}
\star & \star\\
0 &\star
\end {array}
\right]
.$$ From here it is not hard to see that $A$ is ideal-reducible if and only if $A^T$ is ideal-reducible. If there exists a permutation matrix $P$ such that $P^TAP$ is upper-triangular, then $A$ is said to be {\it ideal-triangularizable}.
If $A$ is not ideal-reducible, then $A$ is said to be {\it ideal-irreducible}. At this point we would like to stress out that other authors use different terms for ideal-reducibility and ideal-triangularizability. For instance, in \cite{AbAp} and \cite{RadRos}, the terms that authors use are reducible and decomposable, respectively. In \cite{RadRos} the authors also use {\it complete decomposability} instead of ideal-triangularizability.

It is easy to see that ideal-irreducible positive matrices are strictly positive. If $A$ is an ideal-irreducible positive matrix, then by the classical Perron-Frobenius theory (see \cite[Theorem 8.26]{AbAp}) the spectral radius $r(A)$ is positive, and the eigenspace corresponding to $r(A)$ is one-dimensional and  spanned by a strictly positive eigenvector.

By a {\it cycle of order} $n$ we mean a $n\times n$  matrix $C_n$ defined by $C_ne_j=e_{j-1}$ for all $2\leq j\leq n$ and $C_ne_1=e_n$ where $\{e_1,\ldots,e_n\}$ is the set of all standard basis vectors of the space $\mathbb R^n$ $(\mathbb C^n)$. It is not difficult to see that the $n\times n$ cycle $C_n$ is ideal-irreducible. We denote by $J_n$ the nilpotent $n\times n$ Jordan block.

For details not explained throughout the text regarding the lattice structure of $\mathbb R^n$ and positive matrices we refer the reader to \cite{AbAp}.

We conclude this section with the following lemma in which we prove that positive matrices $A$ and $B$ with a positive commutator $AB-BA$ commute whenever at least one of the matrices $A$ and $B$ is ideal-irreducible. The case when $A$ is ideal-irreducible was considered in the proof of \cite[Theorem 2.1]{BDFRZ}. For the sake of completness we provide a complete proof that also covers the case when the matrix $B$ is ideal-irreducible. We would also like to refer the reader to \cite[Theorem 2.2]{BDFRZ} and \cite[Corollary 3.5]{Gao} for the infinite-dimensional extensions of this result.

\begin {lemma}\label{komutiranje}
Let $A$ and $B$ be $n\times n$ matrices with a positive commutator $AB-BA$. If at least one of the matrices $A$ and $B$ is positive and ideal-irreducible, then $AB=BA.$
\end {lemma}

\begin {proof}
Suppose first that the matrix $A$ is positive and ideal-irreducible. Then  we have $r:=r(A)>0$, and there exist strictly positive vectors $x$ and $y$ such that
$Ax=rx$ and $A^Ty=ry$,  from where the following identity follows $$y^T(AB-BA)x=(A^Ty)^TBx-y^TB(Ax)=ry^TBx-ry^TBx=0.$$
Since $y$ is strictly positive and $AB-BA$ is positive, we first conclude $(AB-BA)x=0$. If $z$ is an arbitrary vector in $\mathbb R^n$, then there exists $\lambda\geq 0$ such that $|z|\leq \lambda x$. From here we obtain
$$|(AB-BA)z|\leq (AB-BA)|z|\leq \lambda (AB-BA)x=0.$$
Since $(AB-BA)z=0$ for each vector $z\in \mathbb{R}^n$, we conclude that $AB=BA.$

If $B$ is ideal-irreducible and positive, then by transposing the inequality $AB-BA\geq 0$ we obtain $B^TA^T-A^TB^T\geq 0$. Since $B^T$ is ideal-irreducible and positive as well, by the first case of the proof we have that $A^T$ and $B^T$ commute. This immediately implies $AB=BA$.
\end {proof}

\section {General case}

In this section we answer Question \ref{vprasanje} with Theorems \ref{dimenzija supercentralizatorja} and \ref{doseg meje}. First we observe that the answer to the special case when one of the matrices is ideal-irreducible follows immediately from Theorem  \ref{Gerstenhaber} and Lemma \ref{komutiranje}.

\begin{corollary}\label{pozitivni Gerstenhaber}
Let $A$ and $B$ be semi-commuting $n\times n$ matrices. If one of the matrices $A$ and $B$ is positive and ideal-irreducible, then the unital algebra generated by $A$ and $B$ is at most $n$-dimensional.
\end{corollary}

The conclusion in Corollary \ref{pozitivni Gerstenhaber} is the same as in Gerstenhaber's theorem, since semi-commuting matrices actually commute when one of the matrices is positive and ideal-irreducible.  If neither $A$ nor $B$ is ideal-irreducible, then, in general, $A$ and $B$ do not commute, and the dimension of the algebra generated by $A$ and $B$ can be greater than $n$.
We will see that the answer still depends on the structure of matrices $A$ and $B$. The following theorem considers the most general case. The dimension of the algebra generated by two positive matrices $A$ and $B$ substantially increases if we replace the assumption that $A$ and $B$ commute with the assumption that $A$ and $B$ semi-commute.

\begin{theorem}\label{dimenzija supercentralizatorja}
Let $A$ and $B$ be semi-commuting positive $n\times n$ matrices. Then the unital algebra generated by $A$ and $B$ is at most $\frac{n(n+1)}{2}$-dimensional. 
\end{theorem}

\begin{proof}
Let us denote by $\mathcal A$ the unital algebra generated by $A$ and $B$.
Let us also denote by $C$ the matrix $A+B$, and let $\mathcal C:\,\{0\}=\mathcal{J}_0\subset \mathcal{J}_1\subset  \cdots \subset \mathcal{J}_k=\mathbb R^n$ be a maximal chain of ideals invariant under $C$ for some $1\leq k\leq n$.

If $k=1$, then $C$ is ideal-irreducible. Since $AC-CA=AB-BA$, $A$ and $C$ semi-commute, so that $A$ and $C$ actually commute by Lemma \ref{komutiranje}. Therefore, $A$ and $B$ commute and the dimension of $\mathcal A$ is at most $n$ by Theorem \ref{Gerstenhaber}.

Suppose now that $k\geq 2$. For every $2\leq j\leq k$, let $\mathcal{K}_j$ be the ideal of $\mathbb R^n$ such that $\mathcal{J}_{j-1}\oplus \mathcal{K}_j=\mathcal{J}_j.$ Since every ideal in $\mathcal C$ is also invariant under $A$, with respect to the decomposition
\begin {equation}\label{dekompozicija}
\mathbb R^n=\mathcal{J}_1\oplus \mathcal{K}_2\oplus \cdots\oplus \mathcal{K}_k
\end {equation}
the matrices $A$ and $C$ can be written as
$$A=\left[
\begin {array}{ccccc}
A_{11} & \cdots&\cdots& \cdots& A_{1k}\\
0 & A_{22} &\cdots & \cdots &A_{2k}\\
\vdots &\ddots &\ddots & &\vdots \\
\vdots& &\ddots &\ddots &\vdots \\
0 & \cdots & \cdots & 0  & A_{kk}
\end {array}
\right] \qquad \textrm{and}\qquad
C=\left[
\begin {array}{ccccc}
C_{11} & \cdots&\cdots& \cdots& C_{1k}\\
0 & C_{22} &\cdots & \cdots &C_{2k}\\
\vdots &\ddots &\ddots & &\vdots \\
\vdots& &\ddots &\ddots &\vdots \\
0 & \cdots & \cdots & 0  & C_{kk}
\end {array}
\right]
$$ with positive blocks.
Therefore, with respect to the decomposition (\ref{dekompozicija}) the matrix $B$ can be written as
$$B=C-A=\left[
\begin {array}{ccccc}
B_{11} & \cdots&\cdots& \cdots& B_{1k}\\
0 & B_{22} &\cdots & \cdots &B_{2k}\\
\vdots &\ddots &\ddots & &\vdots \\
\vdots& &\ddots &\ddots &\vdots \\
0 & \cdots & \cdots & 0  & B_{kk}
\end {array}
\right],$$
where $B_{ij}=C_{ij}-A_{ij}\geq 0$ for all $1\leq i\leq j\leq k.$ Furthermore, maximality of the chain $\mathcal C$ implies ideal-irreducibility of  diagonal blocks $C_{ii}$ for all $1\leq i\leq k$, so that by Lemma \ref{komutiranje} the diagonal blocks $A_{ii}$ and $C_{ii}$ commute for all $1\leq i\leq k.$
Since $AC-CA=AB-BA$, we have that $A_{ii}$ and $B_{ii}$ commute for all $1\leq i\leq k.$
 By applying the fact that every commuting family of complex matrices is simultaneously triangularizable, $A$ and $B$ are upper-triangular in some basis of the space $\mathbb C^n$. From this we conclude that the algebra  $\mathcal A$ is at most $\frac{n(n+1)}{2}$-dimensional.
\end {proof}

From Corollary \ref{pozitivni Gerstenhaber} and the proof of Theorem \ref{dimenzija supercentralizatorja} we can derive that if
$$\mathcal C:\qquad \{0\}=\mathcal{J}_0\subset \mathcal{J}_1\subset  \cdots \subset \mathcal{J}_k=\mathbb R^n$$ is a maximal chain
of ideals invariant under both $A$ and $B$, then the dimension of the unital algebra generated by $A$ and $B$ is at most
$$n+\sum_{i=1}^{k-1}\sum_{j=i+1}^k n_i n_j$$ where $n_i=\dim \mathcal{J}_i-\dim \mathcal{J}_{i-1}$ for all $i=1,\ldots,k.$ This bound is smaller than $\frac{n(n+1)}{2}$ if and only if $A+B$ is not ideal-triangularizable, or equivalently, the pair $\{A,B\}$ is not ideal-triangularizable.

In the following theorem we prove that the upper bound from Theorem \ref{dimenzija supercentralizatorja} can be attained.

\begin {theorem}\label{doseg meje}
For every $n\in \mathbb N$ there exists a positive diagonal matrix $B$ such that $J_nB-BJ_n\geq 0$ and the dimension of the unital algebra generated by $J_n$ and $B$ is precisely $\frac{n(n+1)}{2}$.
\end {theorem}

\begin {proof}
Let $B$ be a diagonal matrix with strictly increasing positive entries on the diagonal and let $\mathcal{A}$ be the unital algebra generated by $J_n$ and $B$. It is clear that $J_n$, $B$ and $J_nB-BJ_n$ are positive. Let us prove by induction on $j-i$ that all the matrices $E_{ij}$ where $i\leq j$ are in $\mathcal A$. Since the diagonal entries of $B$ are distinct, it is a standard result that there exist polynomials $p_{i}$ $(1\leq i\leq n)$ such that $p_i(B)=E_{ii}.$ Consequently, $E_{ii}\in \mathcal{A}$ for each $1\le i\le n$.

Assume now that $j-i>0$ and that $E_{i'j'}\in \mathcal{A}$ for all pairs $(i',j')$ with $i'\le j'$ and $j'-i'<j-i$. A direct computation shows that
\begin {align*}
\left[\sum\limits_{m=1}^iE_{m,m+j-i-1},J_n\right]&
=\left[\sum\limits_{m=1}^ie_me_{m+j-i-1}^T,\sum _{k=1}^{n-1}e_ke_{k+1}^T\right]\\
&=\sum_{m=1}^i e_me_{m+j-i}^T-\sum_{m=2}^ie_{m-1}e_{m+j-i-1}^T=E_{ij}.
\end {align*}
By the assumption each matrix $E_{m,m+j-i-1}$ $(1\le m\le i)$ belongs to $\mathcal{A}$, therefore $E_{ij}\in \mathcal{A}$.
\end{proof}

In Theorem \ref{doseg meje} we have seen that the algebra generated by semi-commuting positive matrices $B$ and $J_n$ is the largest possible if $B$ is taken to have a lot of invariant ideals. We will see in Section \ref{CompanionMatrix} that every positive matrix that semi-commutes with a Jordan block is upper-triangular.
We conclude this section with the following question.

\begin {question}
Given $n\in \mathbb N$ and $n\leq k\leq \frac{n(n+1)}{2}$, do there exist positive $n\times n$ matrices $A$ and $B$ with a positive commutator $AB-BA$ such that the dimension of the unital algebra generated by $A$ and $B$ is precisely $k$?
\end {question}

\section {Cycles and permutations}

In this section we consider Question \ref{vprasanje} when one of the matrices $A$ and $B$ is a permutation matrix.  In the following lemma which follows easily from Corollary \ref{pozitivni Gerstenhaber} we first consider the case when permutation matrices are cycles.

\begin{lemma}\label{cikel}
Let $C_n$ be a cycle of order $n$ and let $A$ be a matrix that semi-commutes with $C_n$. Then the algebra generated by $A$ and $C_n$ is $n$-dimensional.
\end{lemma}

\begin{proof}
The lemma follows directly from Corollary \ref{pozitivni Gerstenhaber}, together with the well-known fact that the powers $I,C_n,C_n^2,\ldots ,C_n^{n-1}$ are linearly independent.
\end{proof}

In particular, Lemma \ref{cikel} implies that every $n\times n$ matrix that semi-commutes with the cycle $C_n$ of order $n$ is a polynomial in $C_n.$

We proceed with the case when a given matrix intertwines two cycles of possibly different sizes.

\begin{lemma}\label{intertwiner}
For a $m\times n$ matrix $A$ the following statements hold.
\begin {enumerate}
\item [(a)] $C_mA=AC_n$ if and only if $a_{i,j-1}=a_{i+1,j}$ for all $1\le i\le m$ and $1\le j\le n$ where the indices $i+1$ and $j-1$ are taken modulo $m$ and $n$, respectively.
\item [(b)] If $C_mA\ge AC_n$, then $a_{i,j-1}=a_{i+1,j}$ for all $1\le i\le m$ and $1\le j\le n$ where the indices $i+1$ and $j-1$ are taken modulo $m$ and $n$, respectively. In particular, $C_mA=AC_n$.
    \end {enumerate}
\end{lemma}

In other words, $A$ is a Toeplitz matrix such that for each $i$ the diagonals containing $a_{mi}$ and $a_{1,i+1}$ (respectively $a_{in}$ and $a_{i+1,1}$) are the same.

\begin {proof}
The proof of (a) can be given by a direct calculation. To see (b), note first that the inequality $C_mA\ge AC_n$ is equivalent to $a_{i+1,j}\ge a_{i,j-1}$ for all $1\le i\le m$ and $1\le j\le n$ where the indices are taken modulo $m$ and $n$. However, if we are moving along the diagonals, we come to the beginning after finitely many steps. Therefore the above inequalities are actually equalities. Moreover, by applying (a) we obtain $C_mA=AC_n.$
\end {proof}

\begin{corollary}\label{dim_intertwiner}
The dimension of the space of all $m\times n$ matrices $A$ satisfying $C_mA=AC_n$ is equal to the greatest common divisor of $m$ and $n$.
\end{corollary}
\begin{sketch}
It follows from Lemma \ref{intertwiner} that $a_{ij}=a_{i+x(\MOD m), j+x(\MOD n)}$ for each $1\le i\le m$, each $1\le j\le n$ and each integer $x$. Let us denote by $d$ and $v$ the greatest common divisor and the least common multiple of $m$ and $n$, respectively. For $j=1,\ldots,d$ we define $u_j=\sum _{x=1}^vE_{1+x(\MOD m),j+x(\MOD n)}$. Then $\{u_1,\ldots ,u_d\}$ is a basis of the vector space of all intertwiners of $C_m$ and $C_n$.
\end{sketch}

\begin{example}
The matrix $A$ satisfying $C_4A=AC_6$ looks like
$$A=\left[
\begin{array}{cccccc}
a&b&a&b&a&b\\b&a&b&a&b&a\\a&b&a&b&a&b\\b&a&b&a&b&a
\end{array}
\right].$$
\end{example}

In the following theorem we consider general permutation matrices.
\begin{theorem}
Let $P$ be a permutation matrix and let $A$ be a matrix such that $A$ and $P$ semi-commute. Then $A$ and $P$ commute and the algebra generated by $A$ and $P$ is at most $n$-dimensional.
\end{theorem}

\begin{proof}
Suppose first that $PA\geq AP$. Then there exists a permutation matrix $Q$ such that
$$\widetilde P:=Q^TPQ=\left[
\begin{array}{ccc}
C_{n_1}\\
&\ddots\\
&&C_{n_k}
\end{array}
\right]$$
for some cycles $C_{n_1},\ldots ,C_{n_k}$, where $n=n_1+\cdots +n_k$. Since the  matrix
$$\widetilde{A}:=Q^TAQ=\left[
\begin{array}{ccc}
\widetilde{A_{11}}&\cdots&\widetilde{A_{1k}}\\
\vdots&\ddots&\vdots\\
\widetilde{A_{k1}}&\cdots&\widetilde{A_{kk}}
\end{array}
\right]$$ satisfies
$$\widetilde P  \widetilde{A}=Q^TPAQ\geq Q^TAPQ=\widetilde A\widetilde P,$$
we have
$C_{n_i}\widetilde{A_{ij}}\geq \widetilde{A_{ij}}C_{n_j}$ for all $1\le i,j\le k$. Applying Lemma \ref{intertwiner} we get $C_{n_i}\widetilde{A_{ij}}=\widetilde{A_{ij}}C_{n_j}$ for all $1\leq i,j\leq k$. This implies that $\widetilde A$ and $\widetilde P$ commute, so that $A$ and $P$ commute as well. If $AP\geq PA$, then $P^TA^T\geq A^TP^T$, so that by the first case $A^T$ commutes with $P^T$. This immediately implies that $A$ and $P$ commute.

Since in both cases the matrices $A$ and $P$ commute, the unital algebra generated by them is at most $n$-dimensional by Theorem \ref{Gerstenhaber}.
\end{proof}

The above proposition does not hold for generalized permutations. Indeed, as we can see from Theorem \ref{doseg meje}, the upper bound for the dimension of the algebra is $\frac{n(n+1)}{2}$.

\section {Companion matrices and Jordan blocks}\label{CompanionMatrix}

In this section we consider Question \ref{vprasanje} in the case when one of the matrices is a  {\it companion matrix}. By a companion matrix we mean a matrix $A$ of the form
$$A=\left[
\begin {array}{ccccc}
0 & 1 & & & \\
0 & 0 & 1 & &\\
\vdots & \vdots &\ddots &\ddots & \\
0 & 0 & \hdots & 0& 1\\
a_0 & a_1 & a_2 & \hdots & a_{n-1}
\end {array}
\right].$$ The minimal polynomial $m$ of the companion matrix $A$ is given by
$m(\lambda)=\lambda^n-a_{n-1}\lambda^{n-1}-\cdots-a_1\lambda-a_0.$ Recall that the algebraic multiplicity of an eigenvalue $\lambda_0$ of a given matrix is equal to the degree of $\lambda_0$ as a zero of its characteristic polynomial. In the case of a companion matrix $A$ zero is an eigenvalue of $A$ with algebraic multiplicity $k$ if and only if $a_0=\cdots=a_{k-1}=0$ and $a_k\neq 0.$

\begin {lemma}\label{unicimo b}
Suppose that positive $n\times n$ matrices $A$ and $B$ are
of the form
$$
A=\left[\begin {array}{cc}
0 & e_1^T\\
0 & A'
\end {array}
\right] \qquad \textrm{and} \qquad
B=\left[\begin {array}{cc}
\beta & b^T\\
b' & B'
\end {array}
\right]
$$ where $e_1,b,b'\in \mathbb R^{n-1}$ and $A'$ is a companion matrix of size $(n-1)\times (n-1)$. If $A$ and $B$ semi-commute, then $b'=0$. In particular, $A$ and $B$ are simultaneously ideal-reducible.
\end {lemma}

\begin {proof}
A direct calculation shows that
$$AB-BA=\left[\begin {array}{cc}
e_1^Tb' & e_1^TB'-\beta e_1^T-b^TA'\\
A'b' & A'B'-B'A'-b'e_1^T
\end {array}
\right].$$

Suppose first that $AB\geq BA$, so that
$A'B'-B'A'\geq b'e_1^T\geq 0.$
Since the matrices $A$ and $B$ are positive, matrices $A'$ and $B'$ are positive as well. By \cite[Theorem 2.1]{BDFRZ} the commutator $A'B'-B'A'$ belongs to the radical of the Banach algebra generated by $A'$ and $B'$, from where it follows that the matrix $(A'B'-B'A')(A')^k$ is  nilpotent for every $k\in \mathbb N_0.$ Since the product of positive matrices is positive, the inequality
$$(A'B'-B'A')(A')^k\geq b'e_1^T(A')^k\geq 0$$ implies that
the matrix $C:=b'e_1^T(A')^k$ is also nilpotent for every $k\in \mathbb N_0$. Since $C$ is of rank one, we have $C^2=0$, so that $e_1^T(A')^kb'=0.$
The fact that $A'$ is a $(n-1)\times (n-1)$ companion matrix gives us $e_1^T(A')^k=e_{k+1}^T$ for all $k=0,\ldots,n-2$, so that
$$e_{k+1}^Tb'=(e_1^T(A')^k)b'=0 \qquad \qquad (0\leq k\leq n-2)$$ implies $b'=0$.

Assume now that $BA\geq AB$. Then the inequalities $e_1^Tb'\leq 0$ and $A'b'\leq 0$ imply $e_1^Tb'=0$ and $A'b'=0$, respectively. If $(b')^T=[b_1',\cdots,b_{n-1}']$, then $e_1^Tb'=0$ implies $b_1'=0$, and the equality $A'b'=[b_2', \cdots,b_{n-1}',\star]^T=[0,\ldots,0]^T$ implies $b_2'=\cdots=b_{n-1}'=0$. In this case again we obtain $b'=0.$
\end {proof}

\begin {proposition}\label{Strukturni}
Let $A$ and $B$ be positive semi-commuting $n\times n$ matrices where $A$ is a companion matrix. If  zero is an eigenvalue of $A$  with algebraic multiplicity $k$ for some $1\leq k\leq n$, then the leading principal submatrix $B_1$ of order $k$ of the matrix $B$ is upper-triangular. Furthermore, if $k<n$, then $B$ is  of the form
$$\left[
\begin{array}{cc}
B_1 & B_2\\
\mathbf{0} & B_3
\end{array}\right]
.$$
\end {proposition}

\begin {proof}
We prove this proposition by induction on the algebraic multiplicity of zero as an eigenvalue of the companion matrix $A$.
If $k=1$, then the conclusion of the proposition follows from Lemma \ref{unicimo b}.

Assume now that $k>1$ and suppose that the proposition holds for all positive semi-commuting square matrices $A'$ and $B'$ where $A'$ is a companion matrix with zero as an eigenvalue of algebraic multiplicity $1\leq k_0\leq k-1.$ By the assumption the matrices $A$ and $B$ are of the form $$A=\left[\begin {array}{cc}
0 & e_1^T\\
0 & A'
\end {array}\right]
 \qquad \textrm{and}\qquad
B=\left[\begin {array}{cc}
\beta & b^T\\
b' & B'
\end {array}
\right]$$ where $A'$ is a positive companion matrix with zero as an eigenvalue of algebraic multiplicity $k-1$. By Lemma \ref{unicimo b} we have $b'=0$. Since positive matrices $A'$ and $B'$ semi-commute and since the companion matrix $A'$ has zero as an eigenvalue of algebraic multiplicity $k-1$, the leading principal minor of order $k-1$ of the matrix $B'$ is upper-triangular by inductive assumption. Furthermore, when $k<n$,  the south-west $(n-k)\times (k-1)$ corner of $B'$ is also zero. To finish the induction step recall that $b'=0$.
\end {proof}

Since zero is an eigenvalue of $J_n$ with algebraic multiplicity $n$, by Proposition \ref{Strukturni} the only  positive matrices that semi-commute with a Jordan block are upper-triangular ones.

\begin {corollary}\label{Jn}
Every positive $n\times n$ matrix which semi-commutes with a Jordan block $J_n$ is upper-triangular.
\end {corollary}

As an application of Proposition \ref{Strukturni} we will determine the upper bound for the dimension of the unital algebra generated by two semi-commuting positive matrices where one of them is a companion matrix.

\begin {theorem}\label{o zgornji meji}
Let $A$ and $B$ be positive semi-commuting $n\times n$ matrices and suppose that $A$ is a companion matrix. If zero is an eigenvalue of the matrix $A$ with algebraic multiplicity $k$ for some $0\leq k\leq n$, then the dimension of the unital algebra generated by $A$ and $B$ is at most $\frac{(2n-k)(k+1)}{2}.$
\end {theorem}

If $k=0$ or $k=n$, the upper bounds given by Theorem \ref{o zgornji meji} are $n$ and $\frac{n(n+1)}{2}$, respectively. The upper bound in the case $k=0$ can be obtained by Corollary \ref{pozitivni Gerstenhaber}, since $k=0$ implies ideal-irreducibility of $A$. When $k=n$, then $A$ is actually $J_n$, and $B$ is upper-triangular by Corollary \ref{Jn}. The upper bound $\frac{n(n+1)}{2}$ can be attained  when $B$ is chosen to be an appropriate diagonal matrix by Theorem \ref{doseg meje}.

\begin {proof}

Suppose first that $k=0$ and write
$$A=\left[
\begin {array}{ccccc}
0 & 1 & & & \\
0 & 0 & 1 & &\\
\vdots & \vdots &\ddots &\ddots & \\
0 & 0 & \hdots & 0& 1\\
a_0 & a_1 & a_2 & \hdots & a_{n-1}
\end {array}
\right].
$$
Since $k=0$ we have $a_0\neq 0.$ We claim that $A$ is ideal-irreducible. To prove this, suppose that $\mathcal J$ is a nonzero ideal invariant under $A$. Then $e_j\in \mathcal J$ for some $1\leq j\leq n.$ If $j>1$,
then $e_l\leq A^{j-l}e_j\in \mathcal J$ for all $1\leq l\leq j.$  In particular $e_1\in \mathcal J$. Since $e_n=\frac{1}{a_0}Ae_1\in \mathcal J$ by the above the ideal $\mathcal J$ contains all standard basis vectors of $\mathbb R^n$ which implies that $A$ is ideal-irreducible.
Since $A$ is ideal-irreducible, the unital algebra generated by $A$ and $B$ is at most $n$-dimensional by Corollary \ref{pozitivni Gerstenhaber}. This finishes the proof in the case when $k=0$.

Suppose now that $0<k\leq n$. By Proposition \ref{Strukturni} the matrix $B$ is of the form
$$B=\left[
\begin{array}{cc}
B_1 & B_2\\
\mathbf{0} & B_3
\end{array}\right]
$$
for some upper-triangular $k\times k$-matrix $B_1$. Then the matrix $A$ can be written in the form
$$A=\left[
\begin{array}{cc}
J_k & \star \\
\mathbf{0} & A'
\end{array}\right]$$
where
$A'$ is an ideal-irreducible companion matrix. Since  $B_1$ is upper-triangular, the unital algebra generated by $J_k$ and $B_1$ is at most $\frac{k(k+1)}{2}$-dimensional. Since $A'$ is ideal-irreducible and semi-commutes with $B_3$, the dimension of the unital algebra generated by $A'$ and $B_3$ is at most $n-k$ by Corollary \ref{pozitivni Gerstenhaber}. Finally we conclude that the dimension of the unital algebra generated by $A$ and $B$ is at most
$$\frac{k(k+1)}{2}+k(n-k)+n-k=\frac{(2n-k)(k+1)}{2}.$$
\end {proof}

\section {Positive idempotents}

In this section we consider the case of idempotent positive $n\times n$ matrices $E$ and $F$ satisfying $EF\geq FE$. By Theorem \ref {dimenzija supercentralizatorja} the dimension of the unital algebra $\mathcal A$ generated by $E$ and $F$ is at most $\frac{n(n+1)}{2}$. This upper bound can be reduced significantly as it is shown in Theorem \ref{12}.
The special cases when $E$ or both $E$ and $E^T$ are strictly positive are considered in Theorem \ref{dim =4,6}. It is proved that the dimension of $\mathcal A$ is at most $6$ in that case. Theorems \ref{dim =4,6} and \ref{12} are proved in the general setting of vector and Banach lattices.

As we have seen in Lemma \ref{komutiranje}, semi-commuting positive matrices commute whenever one of the matrices is ideal-irreducible. If we replace ideal-irreducibility with strict positivity, the following two lemmas show that we still obtain nice algebraic relations  between two semi-commuting positive idempotent linear operators. These relations enable us to significantly reduce the dimension of the unital algebra generated by them.

\begin {lemma}\label{LemaOProduktu1}
Let $E$ and $F$ be positive idempotent linear operators on a vector lattice $L$ which satisfy  $EF\geq FE$.
\begin {enumerate}
\item [(a)] If $E$ is strictly positive, then $EFE=FE$ and $FE$ is an idempotent.
\item [(b)] If $F$ is strictly positive, then $FEF=EF$ and $EF$ is an idempotent.
\end {enumerate}
Assume additionally that $L$ is a Banach lattice.
\begin {enumerate}
\item [(c)] If $F^*$ is strictly positive, then $FEF=FE$ and $FE$ is an idempotent.
\item [(d)] If $E^*$ is strictly positive, then $EFE=EF$ and $EF$ is an idempotent.
\end {enumerate}
\end {lemma}

\begin {proof}
To see (a), first note that the positive operator $E(EF-FE)E$ is zero. Since $E$ is strictly positive, for every positive vector $x\in L$ we have $(EF-FE)Ex=0.$ The fact that the positive cone $L^+$ of $L$ spans $L$ gives us $EFE=FE$. By multiplying the equality $EFE=FE$ with $F$ on the left-hand side we obtain $(FE)^2=FE.$

To prove (c), assume that $F^*$ is strictly positive on the dual Banach lattice $L^*$. Since $E$ and $F$ are idempotents, their adjoint operators $E^*$ and $F^*$ are idempotents as well. Since
$EF\geq FE$ also implies $F^*E^*\geq E^*F^*$, by  (a) we obtain $F^*E^*F^*=E^*F^*$ and that $(FE)^*=E^*F^*$ is an idempotent.
This immediately implies that $FEF=FE$ and that $FE$ is an idempotent.

We omit the proofs of (b) and (d) since they are very similar to the proofs of (a) and (c), respectively.
\end {proof}

\begin {lemma}\label{LemaOProduktu}
Let $E$ and $F$ be positive idempotent linear operators on a vector lattice $L$ which satisfy  $EF\geq FE$.
\begin {enumerate}
\item [(a)] If at least one of $E$ and $F$ is strictly positive, then $(EF-FE)^2=0.$
\end {enumerate}
Assume additionally that $L$ is a Banach lattice.
\begin {enumerate}
\item [(b)] If at least one of $E^*$ and $F^*$ is strictly positive, then $(EF-FE)^2=0.$
\item [(c)] If $E$ and $E^*$ or if $F$ and $F^*$ are strictly positive, then $EF=FE.$
\end {enumerate}
\end {lemma}

\begin {proof}
If  $E$ is strictly positive, then  Lemma \ref{LemaOProduktu1} implies
\begin {align*}
(EF-FE)^2&=EFEF-EFE-FEF+FEFE\\
&=FEF-FE-FEF+FE=0.
\end {align*}
The case when $F$ is strictly positive can be treated similarly.

To prove (b), note  that by Lemma \ref{LemaOProduktu1} we can similarly as in (a) obtain $(F^*E^*-E^*F^*)^2=0$. The latter equality immediately implies $(EF-FE)^2=0$.  The proof of (c) immediately follows from  Lemma \ref{LemaOProduktu1}.
\end {proof}

Surprisingly, Lemma \ref{LemaOProduktu} holds more generally.

\begin {theorem}
Let $E$ and $A$ be linear operators on a vector lattice $L$ which satisfy $EA-AE\geq 0$. Suppose also that $E$ is a positive idempotent.
\begin {enumerate}
\item [(a)] If $E$ is strictly positive, then $AE=EAE$ and $(EA-AE)^2=0$.
\end {enumerate}
Assume additionally that $L$ is a Banach lattice.
\begin {enumerate}
\item [(b)] If $E^*$ is strictly positive, then $EA=EAE$ and $(EA-AE)^2=0$.
\item [(c)] If $E$ and $E^*$ are strictly positive, then $EA=AE$.
\end {enumerate}
\end {theorem}

\begin {proof}
To see (a), assume that $E$ is strictly positive. Since
$E(EA-AE)E=0$, positivity of $EA-AE$ and strict positivity of $E$ imply $EAE-AE=(EA-AE)E=0$.
Applying the identity $EAE=AE$ we obtain
\begin {align*}
(EA-AE)^2&=EAEA-EA^2E-AE^2A+AEAE\\
&=AEA-EA^2E-AEA+A^2E\\
&=A^2E-EA^2E.
\end {align*}
Since $(EA-AE)^2$ is positive as a product of positive operators, and since $E(EA-AE)^2=EA^2E-EA^2E=0,$ strict positivity of $E$ implies $(EA-AE)^2=0.$

Since the proof of (b) is similar to the proofs of (a) and Lemma \ref{LemaOProduktu1}(c), we omit it. From (a) and (b) we immediately obtain (c).
\end {proof}

The upper bound for the dimension of the unital algebra $\mathcal A$ generated by two semi-commuting positive $n\times n$ matrices $A$ and $B$ is at most $\frac{n(n+1)}{2}$ by Theorem \ref{dimenzija supercentralizatorja}.  Furthermore, if one of the matrices $A$ and $B$ is also ideal-irreducible, then the algebra $\mathcal A$ is at most $n$-dimensional by Corollary \ref{pozitivni Gerstenhaber}. When $A$ and $B$ are idempotents,  the upper bound for the dimension of the algebra $\mathcal A$ can be reduced significantly. In the following theorem we first consider the case when one of idempotents or their adjoints is strictly positive.

\begin {theorem}\label{dim =4,6}
Let $L$ be a vector lattice and let $E$ and $F$ be two positive idempotent operators on $L$ which satisfy $EF\geq FE$. Then the following statements hold:
\begin {enumerate}
\item [(a)] If one of the operators $E$ and $F$ is strictly positive, then the dimension of the unital algebra generated by $E$ and $F$ is at most $6$.
\end {enumerate}
Assume additionally that $L$ is a Banach lattice.
\begin {enumerate}
\item [(b)]If one of the operators $E^*$ and $F^*$ is strictly positive, then the dimension of the unital algebra generated by $E$ and $F$ is at most $6$.
\item [(c)] If $E$ and $E^*$ or if $F$ and $F^*$ are strictly positive, then the dimension of the unital algebra generated by $E$ and $F$ is at most $4$.
\end {enumerate}
\end {theorem}

Since the adjoint operator of a bounded linear operator on a normed space corresponds to the matrix transpose in finite dimensions,
Theorem \ref{dim =4,6} can be applied in the case when $L$ is $\mathbb R^n$ ordered componentwise, and the operators in question are matrices.

\begin {proof}
Let us denote by $\mathcal A$ and $\mathcal W$ the unital algebra generated by $E$ and $F$ and the set of all words in $E$ and $F$, respectively.

To see (a), assume that $E$ is strictly positive. The other case is treated similarly. Note first that by Lemma \ref{LemaOProduktu1} we have $EFE=FE.$   We claim that the algebra $\mathcal A$  is equal to the  linear span $\mathcal L$ of the set $\{I,E,F,EF,FE,FEF\}.$ It should be obvious that $\mathcal L\subseteq \mathcal A$. To prove $\mathcal A\subseteq \mathcal L$, it suffices to see that $\mathcal L$ contains words of arbitrary length.
The latter will be proved by induction on the length of words in $\mathcal W$.

The equality $EFE=FE$ implies that $\mathcal L$ contains all words from $\mathcal W$ of length at most $3$. Suppose now that $\mathcal L$ contains all words from $\mathcal W$ of length at most $n$ for some $n\geq 3$, and let $w_{n+1}\in \mathcal W$ be an arbitrary word of length $n+1$. Then $w_{n+1}$ starts with either $E$ or $F$. Suppose that the former happens. Then $w_{n+1}=Ew_n$ for some word $w_n\in \mathcal W$ of length $n$. Since there exist scalars $\lambda_1,\ldots,\lambda_6$ such that
$$w_n=\lambda_1 I+\lambda_2E+\lambda_3F+\lambda_4EF+\lambda_5FE+\lambda_6FEF,$$
we have
$$w_{n+1}=Ew_n=\lambda_1 E+\lambda_2E+\lambda_3EF+\lambda_4EF+\lambda_5EFE+\lambda_6EFEF.$$
From the last line and the equality $EFEF=FEF$ we can conclude that  $w_{n+1}$ can be written as a linear combination of words from $\mathcal W$ of length at most $3$, and since $EFE=FE$, $w_{n+1}$ is contained in $\mathcal L$. The case when the word $w_{n+1}$ starts with $F$ can be dealt similarly. Now it should be obvious that the dimension of $\mathcal A$ is at most $6$.

If one of the operators $E^*$ and $F^*$ is strictly positive, then by (a) the unital algebra generated by $E^*$ and $F^*$, and hence $\mathcal A$,  is at most $6$-dimensional.

To see (c), assume that $E$ and $E^*$ or $F$ and $F^*$ are strictly positive. By Lemma \ref{LemaOProduktu} the operators $E$ and $F$ commute. Therefore, an arbitrary word $w\in \mathcal W$ equals to $EF$ whenever $E$ and $F$ both appear in $w$. This implies that the algebra $\mathcal A$ is linearly spanned by the set $\{I,E,F,EF\}.$
\end {proof}

As we have seen in Theorem \ref{dim =4,6}, the dimension of the unital algebra generated by two positive idempotents $E$ and $F$ on a Banach lattice with $EF\geq FE$ can be bounded by $6$ if one of the operators $E$ and $F$ is strictly positive. Surprisingly, the dimension of the respective algebra is  very small even without the assumption on strict positivity.
First we need some technical preparation.

Let us denote  the set $\{I,E,F,EF\}$ by $\mathcal C_0$. For each $n\in \mathbb N$ we define  $\mathcal C_n$ inductively as $\mathcal C_n=[E,F]\cdot \mathcal C_{n-1}.$ The set $\mathcal G_m$ is defined as
the linear span  of the set $\bigcup\limits_{n=0}^m\mathcal C_n$ for each $m\in\mathbb N_0.$
Furthermore, let us denote by $\mathcal A$ and
$\mathcal F_n$ the unital algebra generated by $E$ and $F$ and the set of all words in $E$ and $F$ of length $n$, respectively. If $n=2m$ for some $m\in \mathbb N$, then
$\mathcal F_n=\{(EF)^m,(FE)^m\}.$ If $n=2m-1$ for some $m\in \mathbb N$, then
$\mathcal F_n=\{(EF)^{m-1}E,(FE)^{m-1}F\}.$ Therefore, the algebra $\mathcal A$ is spanned by $\{I\}\cup \bigcup\limits_{n=1}^\infty \mathcal F_n$ as a vector space.
For $m\in \mathbb N$ we define the vector subspace $\mathcal V_m$ as the linear span of the set  $\{I\}\cup \bigcup\limits_{n=1}^m \mathcal F_n.$

\begin {lemma}\label{GN}
For each $n\in \mathbb N_0$ we have
$$\mathcal G_n=\spam \left(\mathcal V_{2n+1} \cup \{[E,F]^n EF\}\right).$$
\end {lemma}

\begin {proof}
Since we have
$$[E,F]^n\cdot EF=(EF)^{n+1}+t_n(E,F),$$ where
$t_n(E,F)$ is a sum of words in $E$ and $F$ of length at most $2n+1$, we obtain
\begin {equation}\label{keks}
\spam \left(\mathcal V_{2n+1} \cup \{[E,F]^n EF\}\right)=\spam\left(\mathcal V_{2n+1} \cup \{(EF)^{n+1}\}\right).
\end {equation}
By a tedious induction which is left to the reader one can see that $\mathcal G_n$ is spanned by the set of all words of length at most $2n+2$ except $(FE)^{n+1}$.
Therefore, $\mathcal G_n$ is spanned by the set $\{I\}\cup \bigcup\limits_{j=1}^{2n+1} \mathcal F_j\cup\{(EF)^{n+1}\}.$ The conclusion of the lemma now immediately follows.
\end {proof}

Lemma \ref{GN} immediately implies that the algebra $\mathcal A$ is spanned by the set $\bigcup\limits_{n=0}^\infty \mathcal C_n.$

\begin {theorem}\label{12}
Let $E$ and $F$ be positive idempotent linear operators on a vector lattice $L$ with the projection property. If $F$ is order continuous and the commutator $EF-FE$ is positive, then the dimension of the unital algebra generated by $E$ and $F$ is at most $9$.
\end {theorem}

\begin {proof}
 Since $F$ is order continuous, the absolute kernel $\mathcal N(F)$ is a band in $L$. The fact that $L$ has the projection property implies that the lattice $L$ can be decomposed as $L=\mathcal N(F)\oplus \mathcal N(F)^d$. With respect to this decomposition the operators $E$ and $F$ can be decomposed as
$$E=\left[\begin {array}{cc}
E_{11} & E_{12}\\
E_{21} & E_{22}
\end {array}\right] \qquad \textrm{and}\qquad F=\left[\begin {array}{cc}
0 & F_{12}\\
0 & F_{22}
\end {array}\right],$$
where the blocks are positive operators between appropriate vector lattices.

We claim that $F_{22}$ is strictly positive. Indeed, let $x\in \mathcal N(F)^d$ be a positive vector such that
$F_{22}x=0$. Using idempotency of $F$ we obtain $F_{12}=F_{12}F_{22}$. Therefore,
$F_{12}x=0$. This implies that the vector
$\left[\begin {array}{c}
0 \\ x
\end {array}
\right]$ lies in $\mathcal N(F)$. So $x=0$ which proves the claim.

Since
$$EF-FE=\left[\begin {array}{cc}
-F_{12}E_{21} & E_{11}F_{12}+E_{12}F_{22}-F_{12}E_{22}\\
-F_{22}E_{21} & E_{21}F_{12}+E_{22}F_{22}-F_{22}E_{22}
\end {array}\right]\geq 0,$$
we have $F_{22}E_{21}=0$, so that strict positivity of $F_{22}$ implies $E_{21}=0$. Therefore we have
$$EF-FE=\left[\begin {array}{cc}
0& E_{11}F_{12}+E_{12}F_{22}-F_{12}E_{22}\\
0 &E_{22}F_{22}-F_{22}E_{22}
\end {array}\right].$$%
Since $F_{22}$ is strictly positive, by Lemmas \ref{LemaOProduktu1} and \ref{LemaOProduktu}(a) we have
\begin{equation}\label{enakosti}
F_{22}E_{22}F_{22}=E_{22}F_{22} \qquad \textrm{and} \qquad (E_{22}F_{22}-F_{22}E_{22})^2=0.
\end{equation}

We claim that the unital algebra generated by $E$ and $F$ is spanned by the set
$$\mathcal F:=\{I,E,F,EF,[E,F],[E,F]E,[E,F]F,[E,F]EF,[E,F]^2\}.$$
For simplicity of further calculations we denote $E_{11}F_{12}+E_{12}F_{22}-F_{12}E_{22}$ and $E_{22}F_{22}-F_{22}E_{22}$ by $X$ and $K$, respectively.
By a direct calculation using (\ref{enakosti}) one can show that we have
$$
[E,F]E=\left[\begin{array}{cc}
0 & XE_{22}\\
0 & KE_{22}
\end {array}\right],\qquad
[E,F]F=\left[\begin{array}{cc}
0 & XF_{22}\\
0 & 0
\end {array}\right]
$$
and
$$
[E,F]EF=\left[\begin{array}{cc}
0 & XE_{22}F_{22}\\
0 & KE_{22}F_{22}
\end {array}\right].
$$
From the identity $E_{22}F_{22}=K+F_{22}E_{22}$ we obtain
$KE_{22}F_{22}=K^2+KF_{22}E_{22},$ and again by applying (\ref{enakosti})
we actually have
$$
[E,F]EF=\left[\begin{array}{cc}
0 & XE_{22}F_{22}\\
0 & 0
\end {array}\right].
$$
Similarly, we have
$$
[E,F]^2=\left[\begin{array}{cc}
0 & XK\\
0 & 0
\end {array}\right], \qquad
[E,F]^2E=\left[\begin{array}{cc}
0 & XKE_{22}\\
0 & 0
\end {array}\right],
$$
$$
[E,F]^2F=\left[\begin{array}{cc}
0 & XKF_{22}\\
0 & 0
\end {array}\right] \qquad \textrm{and} \qquad
[E,F]^2EF=\left[\begin{array}{cc}
0 & XKE_{22}F_{22}\\
0 & 0
\end {array}\right].
$$
By (\ref{enakosti})
we have $[E,F]^2F=[E,F]^2EF=0$. It is obvious that  $[E,F]^3=0$.

To show that $[E,F]^2E$ is also zero, let $\mathcal{I}$ be the order ideal generated by $E_{22}((\mathcal{N}(F)^d)^+)$. Since the commutator $E_{22}F_{22}-F_{22}E_{22}$ is positive, we have $E_{22}F_{22}x\ge F_{22}E_{22}x$ for each positive vector $x$. Since an arbitrary element of $E_{22}((\mathcal{N}(F)^d)^+)$ is of the form $E_{22}x$ for some positive vector $x\in \mathcal{N}(F)^d$ and since $E_{22}F_{22}x\in E_{22}((\mathcal{N}(F)^d)^+)$ by definition, we get $F_{22}(E_{22}((\mathcal{N}(F)^d)^+))\subseteq \mathcal{I}$, and consequently,
\begin{equation}\label{F_{22}(I)}
F_{22}(\mathcal{I})\subseteq \mathcal{I}.
\end{equation}
Next, idempotency of the operators $F$ and $E$ implies $F_{12}=F_{12}F_{22}$ and $E_{12}=E_{11}E_{12}+E_{12}E_{22}$, so $E_{11}E_{12}E_{22}=0$. In particular,
\begin{equation}\label{E_{11}E_{12}(I)}
E_{11}E_{12}|_{\mathcal{I}}=0.
\end{equation}
Since the ideal $\mathcal I$ is invariant under $F_{22}$ by (\ref{F_{22}(I)}),
using (\ref{E_{11}E_{12}(I)}) we obtain $E_{11}E_{12}F_{22}E_{22}=0$.
Clearly, $E_{11}XE_{22}=E_{11}E_{12}F_{22}E_{22}$, so that
$E_{11}X|_{\mathcal{I}}=0$.
Using (\ref{F_{22}(I)}) and (\ref{E_{11}E_{12}(I)}) on the latter equality
yields
\begin{equation}\label{zozitve na I}
E_{11}F_{12}|_{\mathcal{I}}=E_{11}F_{12}E_{22}|_{\mathcal{I}}.
\end{equation}
Using the equality $E_{12}=E_{11}E_{12}+E_{12}E_{22}$ we compute
\begin{eqnarray*}
XKE_{22}&=&(E_{11}F_{12}+E_{12}F_{22}-F_{12}E_{22})(E_{22}F_{22}-F_{22}E_{22})E_{22}\\
&=&(E_{11}F_{12}E_{22})F_{22}E_{22}-(E_{11}F_{12})F_{22}E_{22}\\
&+&E_{11}E_{12}F_{22}(E_{22}F_{22}-F_{22}E_{22})E_{22}+E_{12}(E_{22}F_{22}E_{22}F_{22}E_{22}-E_{22}F_{22}E_{22})\\
&-&F_{12}E_{22}(E_{22}F_{22}-F_{22}E_{22})E_{22}.
\end{eqnarray*}
Using (\ref{F_{22}(I)}) and (\ref{zozitve na I}) we get
$$(E_{11}F_{12}E_{22})F_{22}E_{22}-(E_{11}F_{12})F_{22}E_{22}=0.$$
By (\ref{F_{22}(I)}) the ideal $\mathcal{I}$ is invariant under $F_{22}$, and since $E_{22}$ is idempotent, $\mathcal{I}$
is also invariant under $E_{22}$. Therefore (\ref{E_{11}E_{12}(I)}) implies
$$E_{11}E_{12}F_{22}(E_{22}F_{22}-F_{22}E_{22})E_{22}=0.$$
Next,
$$E_{12}(E_{22}F_{22}E_{22}F_{22}E_{22}-E_{22}F_{22}E_{22})=0$$
by (\ref{enakosti}). Finally, the identity
$$F_{12}E_{22}(E_{22}F_{22}-F_{22}E_{22})E_{22}=0$$
should be obvious. Therefore $[E,F]^2E=0$.

Since the algebra $\mathcal A$ is spanned by $\bigcup\limits_{n=0}^\infty\mathcal C_n$,
the fact that we have $[E,F]^3=0$
implies that the algebra $\mathcal A$ is spanned by $\mathcal C_0\cup \mathcal C_1\cup \mathcal C_2.$ Finally, by applying $[E,F]^2E=[E,F]^2F=[E,F]^2EF=0$ we conclude that the algebra $\mathcal A$ is spanned by the set $\mathcal F$. This immediately implies that the dimension of $\mathcal A$ is at most $9$.
\end {proof}

Since the space $\mathbb R^n$ ordered componentwise has the projection property, and matrices are order continuous on $\mathbb R^n$, we get an immediate corollary for matrices.

\begin {corollary}\label{122}
Let $E$ and $F$ be positive idempotent $n\times n$ matrices with a positive commutator $EF-FE$. Then the dimension of the unital algebra generated by $E$ and $F$ is at most $9$.
\end {corollary}

As the following example shows the upper bound in Theorem \ref{12} can be obtained.

\begin {example}\label{9max}
Let
$$E=\left[\begin {array}{ccccccc}
1& 0& 1& 0& 0& 0& 0\\
0& 0& 0& 0& 0& 1& 0\\
0& 0& 0& 0& 0& 0& 0\\
0& 0& 0& 0& 0& 0& 0\\
0& 0& 0& 0& 1& 0& 0\\
0& 0& 0& 0& 0& 1& 0\\
0& 0& 0& 0& 0& 1& 0
\end {array}\right] \qquad \textrm{and}\qquad
F=\left[\begin {array}{ccccccc}
0& 0& 0& 0& 0& 0& 0\\
0& 0& 0& 0& 0& 0& 0\\
0& 0& 1& 0& 0& 0& 0\\
0& 0& 0& 0& 0& 0& 0\\
0& 0& 0& 0& 0& 0& 0\\
0& 0& 0& 1& 1& 1& 0\\
0& 0& 0& 0& 0& 0& 1
\end {array}\right].
$$
A direct calculation shows that the commutator $EF-FE$ is positive, and that the set
$$\{I,E,F,EF,[E,F],[E,F]E,[E,F]F,[E,F]EF,[E,F]^2\}$$ which spans the unital algebra generated by $E$ and $F$ is linearly independent.
\end {example}

If we only assume that  idempotent matrices $E$ and $F$ satisfy $EF\geq FE\geq 0$, then their commutator $EF-FE$ is nilpotent by \cite[Theorem 2.3]{BDFRZ}. Therefore we study also idempotents with nilpotent commutators without any positivity assumptions.
The upper bound for the algebra generated by two  idempotent elements of an associative algebra with nilpotent commutator is closely related to the nil-index of their commutator.

\begin{theorem}\label{nilkomutator}
If idempotents $E$ and $F$ in an associative algebra satisfy $(EF-FE)^k=0$ for some $k\in \mathbb N$, then the dimension of the unital algebra generated by $E$ and $F$ is at most $4k$.
\end{theorem}

\begin {proof}
Recall that we denoted by $\mathcal A$ and $\mathcal F_n$ the unital algebra generated by $E$ and $F$ and the set of all words in $E$ and $F$ of length $n$, respectively.

We first prove that $\mathcal{V}_{2k}=\mathcal{V}_{2k+1}$. Since $\mathcal V_{2k+m_1}\subseteq \mathcal V_{2k+m_2}$ whenever $0\leq m_1\leq m_2$, we conclude that
$\mathcal V_{2k}\subseteq \mathcal V_{2k+1}$. To prove the opposite inclusion we need to see that $\mathcal F_{2k+1}\subseteq \mathcal V_{2k}.$
After expanding the identity $(EF-FE)^k=0$ and rearranging its terms, we can write
\begin {equation}\label{razpis}
(EF)^k=(-1)^{k+1}(FE)^k+s_{k-1}(E,F)
\end {equation}
 where
$s_{k-1}(E,F)$ is a sum of some words in $E$ and $F$ of length at most $2k-1$. By multiplying the identity (\ref{razpis}) by $F$ on the left-hand side we obtain
$$(FE)^kF=F(EF)^k=(-1)^{k+1}(FE)^k+F\,s_{k-1}(E,F).$$ Since $F\,s_{k-1}(E,F)$ is  a sum of some words in $E$ and $F$ of length at most $2k$, we have $(FE)^kF\in \mathcal V_{2k}.$ Similarly we can see $(EF)^kE\in \mathcal V_{2k}$, so that $\mathcal V_{2k}=\mathcal V_{2k+1}.$

We prove by induction that  $\mathcal V_{2k}=\mathcal V_{2k+m}$ for all $m\in\mathbb N_0.$ For $m=0$ the equality obviously holds. Suppose now that for some $m\in \mathbb N_0$ we have $\mathcal V_{2k}=\mathcal V_{2k+m}$. Then $\mathcal V_{2k}=\mathcal V_{2k+j}$ for every $0\leq j\leq m$, so that
$$\mathcal F_{2k+m+1}\subseteq \mathcal V_{2k+1} \cdot \mathcal F_m\subseteq
\mathcal V_{2k} \cdot \mathcal V_m\subseteq \mathcal V_{2k+m}=\mathcal V_{2k}.$$
By the definition of $\mathcal V_{2k+m+1}$ we therefore have $\mathcal V_{2k+m+1}\subseteq \mathcal V_{2k}$ which finishes the induction step.
Since the equality (\ref{razpis}) implies that the set $\{I,(FE)^k\}\cup  \bigcup\limits_{n=1}^{2k-1}\mathcal F_n$ spans $\mathcal A$, the dimension of $\mathcal A$ is at most $2+2(2k-1)=4k.$
\end {proof}

The above theorem immediately implies that a unital algebra generated by two simultaneously triangularizable idempotent $n\times n$ matrices is at most $4n$-dimensional. However, this bound is not best possible. Gaines, Laffey and Shapiro \cite{GainesLaffeyShapiro} proved the following more general result.

\begin {theorem}[Gaines, Laffey, Shapiro]
A unital algebra generated by two $n\times n$ matrices with quadratic minimal polynomials is at most $2n$-dimensional if $n$ is even, and at most $(2n-1)$-dimensional if $n$ is odd.
\end {theorem}

The upper-bound $2n-1$ can be attained (for each $n$) even if we additionally assume that the matrices are idempotent and simultaneously triangularizable. Indeed, let
$$E=\left[ \begin{array}{ccccc}
1\\
& 0\\
& & 1\\
& & & 0\\
& & & & \ddots
\end{array}
\right]$$ and $F$ be an $n\times n$ matrix defined by the following properties:
\begin{itemize}
  \item The diagonal equals $(1,0,1,0,1,\ldots)$.
  \item The first super-diagonal consists of all ones.
  \item The $2j$-th super-diagonal equals $(-C_{j-1},C_{j-1},-C_{j-1},\ldots)$ where $C_{j-1}=\frac{1}{j}{2j-2\choose j-1}$ is the $(j-1)$-th Catalan number for $j=1,\ldots,\lfloor \frac{n-1}{2}\rfloor$.
  \item Other entries are zero.
\end{itemize}
It can be verified that $F$ is an idempotent and that the commutator $EF-FE$ has zero entries everywhere except on the first super-diagonal which equals $(1,-1,1,-1,1\ldots)$. Therefore the matrices $$I,E,[E,F],[E,F]E,[E,F]^2,[E,F]^2E,\ldots ,[E,F]^{n-1}$$
are linearly independent.
It can be also verified that
$$F=E+[E,F]-2[E,F]E+\sum_{j=1}^{\lfloor \frac{n}{2}\rfloor}(-1)^{j-1}C_{j-1}[E,F]^{2j}(2E-I)$$ and
$$EF=E+[E,F]-[E,F]E+\sum_{j=1}^{\lfloor \frac{n}{2}\rfloor}(-1)^{j-1}C_{j-1}[E,F]^{2j}E.$$ Therefore,
the unital algebra generated by $E$ and $F$ is $(2n-1)$-dimensional, by Lemma \ref{GN}. We do not know whether the bound $2n$ can be obtained for even $n$.

Combining \cite{GainesLaffeyShapiro} with \cite[Theorem 2.3]{BDFRZ} and \cite[Theorem 2.3.10]{RadRos} we obtain the following corollary. For the original proof of \cite[Theorem 2.3.10]{RadRos} and its extension to quadratic operators we refer the reader to \cite{Szep} and \cite{RadRos2}, respectively.

\begin {corollary}\label{4n drzi}
Let $E$ and $F$ be real idempotent $n\times n$ matrices that satisfy $EF\geq FE\geq 0$. Then the dimension of the unital algebra generated by $E$ and $F$ is at most $2n$ for even $n$ and $2n-1$ for odd $n$, and the pair $\{E,F\}$ is triangularizable.
\end {corollary}

The dimension of the unital algebra generated by positive idempotent $n\times n$ matrices $E$ and $F$ with a positive commutator can be $2n-1$ when $n=3$. Indeed,
if
$$E=\left[\begin {array}{ccc}
0&1&0\\
0&1&0\\
0&0&0\end {array}
\right] \qquad \textrm{and}\qquad
F=\left[\begin {array}{ccc}
0&0&0\\
0&1&1\\
0&0&0\end {array}
\right],$$ then
$$EF-FE=\left[\begin {array}{ccc}
0&1&1\\
0&0&1\\
0&0&0\end {array}
\right].$$ Now it is not hard to see that
$I,E,[E,F],[E,F]E$ and $[E,F]^2$ are linearly independent.

For general $n$ we do not know whether the unital algebra generated by real idempotent $n\times n$ matrices $E$ and $F$ with $EF\geq FE\geq 0$ can be $2n$-dimensional for even $n$ and $2n-1$-dimensional for odd $n$.

\begin {question}
What is the precise upper bound for the dimension of the unital algebra generated by real idempotent $n\times n$ matrices $E$ and $F$ which satisfy $EF\geq FE\geq 0.$
\end {question}

However, in infinite dimensions the algebra generated by two simultaneously triangularizable idempotents can be infinite-dimensional as the following example shows.

\begin{example}
Let $V$ be the Volterra operator on $L^2[0,1]$ defined by
$$(Vf)(x)=\int_0^x f(t)dt.$$ By a direct calculation one can verify that the operators
$$E=\left[\begin {array}{cc}
I & 0\\
0 & 0
\end {array}
\right]\qquad \textrm{and}\qquad
F=\left[\begin {array}{cc}
V & V\\
I-V & I-V
\end {array}
\right]$$
are idempotents
on $L^2[0,1]\oplus L^2[0,1].$
The operator $V$ is quasinilpotent by \cite[Example 7.2.5]{RadRos}. On the other hand, it is not nilpotent since the norm of $V^n$ can be bounded from below by $\frac{1}{2n!}$. The proof of the latter fact can be found in \cite{Kershaw}, see also \cite{Eveson} for its generalization. Consequently, the Volterra operator is not algebraic.

The operators $E$ and $F$ are simultaneously triangularizable in the sense of Banach spaces  by \cite[Example 2]{RadRos2}. The commutator $$EF-FE=\left[\begin {array}{cc}
0 & V\\
V-I & 0
\end {array}
\right]$$ is not nilpotent since
$$(EF-FE)^{2n}=\left[\begin {array}{cc}
(V^2-V)^n & 0\\
0 & (V^2-V)^n
\end {array}
\right]$$
and $V$ is not algebraic.

Since
$$(EF)^n=\left[\begin {array}{cc}
V^n & V^n\\
0 & 0
\end {array}
\right]$$ and $V$ is not algebraic,
the unital algebra generated by $E$ and $F$ is infinite-dimensional.
\end{example}

We conclude this section with the following interesting observation.

\begin {proposition}
Let $E$ and $F$ be simultaneously triangularizable idempotent $n\times n$ matrices. Then there exists a basis of $\mathbb C^n$ such that in this basis $E$ is diagonal and $F$ is upper-triangular.
\end {proposition}

\begin {proof}
Without any loss of generality we may assume that $E$ and $F$ are already upper-triangular. Since $E$ is an idempotent, its diagonal entries are either zero or one. Therefore, we can write
$$E=\left[\begin{array}{cccc}
E_{11} & E_{12} &\cdots & E_{1k}\\
0 & E_{22} & \cdots & E_{2k}\\
\vdots & & \ddots &\vdots \\
0& \cdots & 0&E_{kk}
\end{array}\right] \qquad \textrm{and}\qquad
F=\left[\begin{array}{cccc}
F_{11} & F_{12} &\cdots & F_{1k}\\
0 & F_{22} & \cdots & F_{2k}\\
\vdots & & \ddots &\vdots \\
0& \cdots & 0&F_{kk}
\end{array}\right]$$
for some $1\leq k\leq n$ where the spectrum of each $E_{ii}$ is either $\{0\}$ or $\{1\}$, and the spectra of two consecutive diagonal blocks of $E$ are disjoint.  Since every idempotent matrix with spectrum $\{0\}$ or $\{1\}$ is the zero matrix or the identity matrix, respectively, the matrices $E_{ii}$ are either zero or identity matrices of appropriate sizes. By exchanging $E$ by $I-E$ (if needed) we can without any loss of generality assume that $E_{11}$ is the identity matrix.
Let us denote
$$P=\left[
\begin {array}{ccccc}
I & E_{12} \\
& I &-E_{23} \\
& & I &E_{34}\\
& & & I & \ddots \\
& & & & \ddots
\end {array}
\right].$$
Since $E$, $F$ and $P$ are upper-triangular, so are $PEP^{-1}$ and $PFP^{-1}.$ Furthermore,  the diagonal of the matrix $PEP^{-1}$ is the same as the diagonal of $E$, while the first super-diagonal of $PEP^{-1}$ is zero. If $k\geq 3$, then the idempotency of $E$ implies that the second super-diagonal of $PEP^{-1}$ is zero.
Therefore, we may assume from the start that the first two super-diagonals of $E$ are zero.

Suppose that the first $2j$ super-diagonals of $E$ are zero for some $j$, and let us denote by $P$ the block matrix with the following properties:
\begin {itemize}
\item The sizes of the blocks are the same as the sizes of the corresponding blocks of the matrices $E$ and $F$.
\item The diagonal blocks are identity matrices.
\item The $(2j+1)$-st super-diagonal equals $(E_{1,2j+2},-E_{2,2j+3},E_{3,2j+4},\ldots)$.
\item Other blocks are zero.
\end {itemize}
Since $E$, $F$ and $P$ are upper-triangular, so are $PEP^{-1}$ and $PFP^{-1}.$ By a direct calculation we can see that the diagonal of the matrix $PEP^{-1}$ is the same as the diagonal of $E$, that the first $2j$ super-diagonals of $PEP^{-1}$ are zero and that its $(2j+1)$-st super-diagonal is also zero.
If $k>2j+2$, then the idempotency of $PEP^{-1}$ implies that its $(2j+2)$-nd super-diagonal is zero as well. Therefore, we may assume from the start that the first $2j+2$ super-diagonals of $E$ are zero.

By repeating this procedure we get the required form of the matrices after finitely many steps.
\end {proof}

{\it Acknowledgments.} The authors would like to thank the referee for careful review, in particular for finding a mistake in an earlier version of Example \ref{9max} and for providing the reference \cite{GainesLaffeyShapiro}.

This work was supported in part by grant P1-0222 of Slovenian Research Agency.

\bigskip
		
\noindent
     Marko Kandi\'{c}, Klemen \v Sivic : \\
     Faculty of Mathematics and Physics \\
     University of Ljubljana \\
     Jadranska 19 \\
     1000 Ljubljana \\
     Slovenia \\[1mm]
     E-mails : marko.kandic@fmf.uni-lj.si, klemen.sivic@fmf.uni-lj.si

\end{document}